\documentclass[12p]{amsart}
\usepackage{amssymb}
\usepackage{amsmath}
\usepackage{amsfonts}
\usepackage{geometry}
\usepackage{graphicx}
\usepackage{mathrsfs,amssymb}

\usepackage{hyperref}
\usepackage{cleveref}

\theoremstyle{plain}

\newtheorem{definition}{Definition}

\newtheorem{lemma}{Lemma}

\newtheorem{remark}{Remark}

\newtheorem{theorem}{Theorem}
\numberwithin{equation}{section}

\begin{document}
\title[]{The $L^{2}$ sequential convergence of a solution to the one dimensional, mass-critical NLS above the ground state}

\author{Benjamin Dodson}

\begin{abstract}
In this paper we generalize a weak sequential result of \cite{fan20182} to any non-scattering solutions in one dimension. No symmetry assumptions are required for the initial data.

\end{abstract}
\maketitle

\section{Introduction}
In one dimension, the mass-critical nonlinear Schr{\"o}dinger equation (NLS) is the quintic NLS,
\begin{equation}\label{1.1}
i u_{t} + u_{xx} = \mu |u|^{4} u = \mu F(u), \qquad u(0,x) = u_{0}, \qquad u : I \times \mathbb{R} \rightarrow \mathbb{C}, \qquad \mu = \pm 1.
\end{equation}
$I \subset \mathbb{R}$ is an open interval, $0 \in I$. The case when $\mu = +1$ is the defocusing case, and the case when $\mu = -1$ is the focusing case.

Equation $(\ref{1.1})$ is called mass-critical due to the scaling symmetry. That is, if $u$ solves $(\ref{1.1})$, then for any $\lambda > 0$,
\begin{equation}\label{1.2}
\lambda^{1/2} u(\lambda^{2} t, \lambda x),
\end{equation}
also solves $(\ref{1.1})$ with initial data $\lambda^{1/2} u_{0}(\lambda x)$. The $L^{2}$ norm is preserved under $(\ref{1.2})$. The $L^{2}$ norm, or mass, is also conserved by the flow of $(\ref{1.1})$, if $u$ is a solution to $(\ref{1.1})$ on some interval $I \subset \mathbb{R}$, $0 \in I$, then for any $t \in I$,
\begin{equation}\label{1.3}
M(u(t)) = \int |u(t,x)|^{2} dx = \int |u(0,x)|^{2} dx.
\end{equation}

It is well-known that the local well-posedness of $(\ref{1.1})$ is completely determined by $L^{2}$-regularity. In the positive direction,  \cite{cazenave1989some}, \cite{cazenave1990cauchy} proved that $(\ref{1.1})$ is locally well-posed on some open interval for initial data $u_{0} \in L^{2}(\mathbb{R})$. Furthermore, if $u_{0} \in H_{x}^{s}(\mathbb{R})$ for some $s > 0$,  \cite{cazenave1989some}, \cite{cazenave1990cauchy} proved that $(\ref{1.1})$ was locally well-posed on an open interval $(-T, T)$, where $T(\| u_{0} \|_{H^{s}}) > 0$ depends only on the size of the initial data. Finally,  \cite{cazenave1989some}, \cite{cazenave1990cauchy} proved that there exists $\epsilon_{0} > 0$ such that if $\| u_{0} \|_{L^{2}} < \epsilon_{0}$, then $(\ref{1.1})$ is globally well-posed and scattering.
\begin{definition}[Scattering]\label{d1.1}
A solution to $(\ref{1.1})$ that is global forward in time, that is $u$ exists on $[0, \infty)$, is said to scatter forward in time if there exists $u_{+} \in L^{2}(\mathbb{R})$ such that
\begin{equation}\label{1.4}
\lim_{t \nearrow \infty} \| u(t) - e^{it \partial_{xx}} u_{+} \|_{L^{2}(\mathbb{R})} = 0.
\end{equation}
A solution to $(\ref{1.1})$ that is global backward in time is said to scatter backward in time if there exists $u_{-} \in L^{2}(\mathbb{R})$ such that
\begin{equation}\label{1.5}
\lim_{t \searrow -\infty} \| u(t) - e^{it \partial_{xx}} u_{-} \|_{L^{2}(\mathbb{R})} = 0.
\end{equation}
Equation $(\ref{1.1})$ is scattering for any $u_{0} \in L^{2}(\mathbb{R})$, or for $u_{0}$ in a specified subset of $L^{2}(\mathbb{R})$, if for any $u_{0} \in L^{2}(\mathbb{R})$ or the specified subset of $L^{2}(\mathbb{R})$, there exist $(u_{-}, u_{+}) \in L^{2}(\mathbb{R}) \times L^{2}(\mathbb{R})$ such that $(\ref{1.4})$ and $(\ref{1.5})$ hold, and additionally, $u_{-}$ and $u_{+}$ depend continuously on $u_{0}$.
\end{definition}

The qualitative global behavior for $(\ref{1.1})$ in the defocusing case $(\mu = +1)$ has now been completely worked out. Equation $(\ref{1.1})$ was proved to be globally well-posed and scattering for any initial data in $u_{0} \in L^{2}(\mathbb{R})$, see \cite{dodson2016global}.

In contrast, in the focusing case $(\mu = -1)$, the existence of non-scattering solutions to $(\ref{1.1})$ has been known for a long time, see \cite{glassey1977blowing}. The ground state of equation $(\ref{1.1})$ is
\begin{equation}\label{1.6}
Q(x) = (\frac{3}{\cosh(2x)^{2}})^{1/4}.
\end{equation}
Indeed, the function $Q(x)$ solves the elliptic partial differential equation
\begin{equation}\label{1.10}
Q_{xx} + Q^{5} = Q.
\end{equation}
Therefore, $e^{it} Q(x)$ gives a global solution to $(\ref{1.1})$ in the focusing case that does not scatter in either time direction. Furthermore, if $u(t,x)$ is a solution to $(\ref{1.1})$, then applying the pseudoconformal transformation to $u$,
\begin{equation}\label{1.11}
\frac{1}{t^{1/2}} \bar{u}(\frac{1}{t}, \frac{x}{t}) e^{i \frac{|x|^{2}}{4t}},
\end{equation}
is also a solution to $(\ref{1.1})$. Applying the pseudoconformal transformation to $e^{it} Q(x)$ gives a solution to $(\ref{1.1})$ that blows up in finite time.

Furthermore, the mass $\| Q \|_{L^{2}}$ represents a blowup threshold. In the case when $\| u_{0} \|_{L^{2}} < \| Q \|_{L^{2}}$ and $u_{0} \in H^{1}$, \cite{weinstein1983nonlinear} proved that $(\ref{1.1})$ has a global solution. This follows from conservation laws and the Gagliardo--Nirenberg inequality. A solution to $(\ref{1.1})$ has the conserved quantities mass, $(\ref{1.3})$, energy,
\begin{equation}\label{1.7}
E(u(t)) = \frac{1}{2} \int |u_{x}(t,x)|^{2} dx + \frac{\mu}{6} \int |u(t,x)|^{6} dx = E(u(0)),
\end{equation}
and momentum
\begin{equation}\label{1.8}
P(u(t)) = Im \int \nabla u(t,x) \overline{u(t,x)} dx = P(u(0)).
\end{equation}
When $\mu = +1$, $(\ref{1.7})$ is positive definite, so if $u_{0} \in H^{1}(\mathbb{R})$, then the energy gives an upper bound on $\| u(t) \|_{H^{1}}$ for any $t \in I$, which is enough to prove global well-posedness in the defocusing case. In the focusing case, the Gagliardo--Nirenberg inequality,
\begin{equation}\label{1.9}
\| f \|_{L^{6}(\mathbb{R})}^{6} \leq 3 (\frac{\| f \|_{L^{2}(\mathbb{R})}^{2}}{\| Q \|_{L^{2}(\mathbb{R})}^{2}})^{2} \| \partial_{x} f \|_{L^{2}(\mathbb{R})}^{2},
\end{equation}
and conservation of energy implies an upper bound on $\| u(t) \|_{H^{1}}$ for initial data $u_{0} \in H^{1}$ and $\| u_{0} \|_{L^{2}} < \| Q \|_{L^{2}}$. For initial data $u_{0} \in L^{2}$ satisfying $\| u_{0} \|_{L^{2}} < \| Q \|_{L^{2}}$, where $u_{0}$ need not lie in $H^{1}$, \cite{dodson2015global} proved global well-posedness and scattering.

It is conjectured that $u(t,x) = e^{it} Q(x)$ and its pseudoconformal transformation are the only non-scattering solutions to $(\ref{1.1})$ in the focusing case when $\| u_{0} \|_{L^{2}} = \| Q \|_{L^{2}}$, modulo symmetries of $(\ref{1.1})$. The symmetries of $(\ref{1.1})$ include the scaling symmetry, which has already been discussed $(\ref{1.2})$, translation in space and time,
\begin{equation}\label{1.13}
u(t - t_{0}, x - x_{0}), \qquad t_{0} \in \mathbb{R}, \qquad x_{0} \in \mathbb{R},
\end{equation}
phase transformation,
\begin{equation}\label{1.14}
\forall \theta_{0} \in \mathbb{R}, \qquad e^{i \theta_{0}} u(t, x),
\end{equation}
and the Galilean transformation,
\begin{equation}\label{1.15}
e^{i \frac{\xi_{0}}{2} (x - \frac{\xi_{0}}{2} t)} u(t, x - \xi_{0} t), \qquad \xi_{0} \in \mathbb{R}.
\end{equation}

This conjecture was answered in the affirmative for the focusing, mass-critical problem in all dimensions,
\begin{equation}\label{1.15.1}
i u_{t} + \Delta u = -|u|^{\frac{4}{d}} u, \qquad u(0,x) = u_{0}, \qquad u : I \times \mathbb{R}^{d} \rightarrow \mathbb{C},
\end{equation}
for finite time blowup solutions with finite energy initial data. See \cite{merle1992uniqueness} and \cite{merle1993determination}. This conjecture was also answered in the affirmative for a radially symmetric solution to $(\ref{1.15.1})$ in dimensions $d \geq 4$ that blow up in both time directions, but not necessarily in finite time. 

More recently, \cite{fan20182} proved a sequential convergence result for radially symmetric solutions that may only blow up in one time direction.
\begin{remark}
The pseudoconformal transformation of the solution $e^{it} Q(x)$ is a solution that blows up in one time direction but scatters in the other. By time reversal symmetry, it is possible to assume without loss of generality that the solution blows up forward in time.
\end{remark}
\begin{theorem}\label{t1.1}
Assume that $u$ is a radial solution to the focusing, mass-critical nonlinear Schr{\"o}dinger equation, $(\ref{1.15.1})$, with $\| u_{0} \|_{L^{2}} = \| Q \|_{L^{2}}$ which does not scatter forward in time. Let $(T^{-}(u), T^{+}(u))$ be its lifespan, $T^{-}(u)$ could be $-\infty$ and $T^{+}(u)$ could be $+\infty$. Then there exists a sequence $t_{n} \nearrow T^{+}(u)$ and a family of parameters $\lambda_{\ast, n}$, $\gamma_{\ast, n}$ such that
\begin{equation}\label{1.16}
\lambda_{\ast, n}^{d/2} u(t_{n}, \lambda_{\ast, n} x) e^{-i \gamma_{\ast, n}} \rightarrow Q, \qquad \text{in} \qquad L^{2}.
\end{equation}
\end{theorem}
In fact, \cite{fan20182} proved Theorem $\ref{t1.1}$ for a larger class of initial data, data which is symmetric across $d$ linearly independent hyperplanes. In one dimension, there is no difference between radial initial data and symmetric initial data, but there is in higher dimensions.

In this paper we remove the symmetry assumption in dimension one. In doing so, we must allow for translation and Galilean symmetries, not just scaling and phase transformation symmetries.
\begin{theorem}\label{t1.2}
Assume $u$ is solution to $(\ref{1.1})$ with $\| u_{0} \|_{L^{2}} = \| Q \|_{L^{2}}$ and $\mu = -1$, which does not scatter forward in time. Let $(T^{-}(u), T^{+}(u))$ be its lifespan, $T^{-}(u)$ could be $-\infty$ and $T^{+}(u)$ could be $+\infty$. Then there exists a sequence $t_{n} \nearrow T^{+}(u)$ and a family of parameters $\lambda_{\ast, n}$, $\gamma_{\ast, n}$, $\xi_{\ast, n}$, $x_{\ast, n}$ such that
\begin{equation}\label{1.17}
\lambda_{\ast, n}^{1/2} e^{ix \xi_{\ast, n}} u(t_{n}, \lambda_{\ast, n} x + x_{\ast, n}) e^{-i \gamma_{\ast, n}} \rightarrow Q, \qquad \text{in} \qquad L^{2}.
\end{equation}
\end{theorem}

We can also extend the result of \cite{fan20182} of weak convergence of solutions with mass slightly above the mass of the ground state,
\begin{equation}\label{1.18}
\| Q \|_{L^{2}} < \| u_{0} \|_{L^{2}} \leq \| Q \|_{L^{2}} + \alpha, \qquad \text{for some} \qquad \alpha > 0 \qquad \text{small}.
\end{equation}
\begin{theorem}\label{t1.3}
Assume $u$ is a solution to $(\ref{1.1})$ with $u_{0}$ satisfying $(\ref{1.18})$ and $\mu = -1$, which does not scatter forward in time. Let $(T^{-}(u), T^{+}(u))$ be the lifespan of the solution. Then there exists a sequence of times $t_{n} \nearrow T^{+}(u)$ and a family of parameters $\lambda_{\ast, n}$, $\gamma_{\ast, n}$, $\xi_{\ast, n}$, $x_{\ast, n}$ such that
\begin{equation}\label{1.19}
\lambda_{\ast, n}^{1/2} e^{ix \xi_{\ast, n}} u(t_{n}, \lambda_{\ast, n} x + x_{\ast, n}) e^{-i \gamma_{\ast, n}} \rightharpoonup Q, \qquad \text{weakly in} \qquad L^{2}.
\end{equation}
\end{theorem}

\section{A Preliminary reduction}
The scattering result of \cite{dodson2015global} (Theorem $1.7$) implies that $(\ref{1.1})$ scatters for $\| u_{0} \|_{L^{2}} < \| Q \|_{L^{2}}$, so a non-scattering solution to $(\ref{1.1})$ with $\| u_{0} \|_{L^{2}} = \| Q \|_{L^{2}}$ is a minimal mass blowup solution to $(\ref{1.1})$.
\begin{remark}
A blowup solution is a solution that fails to scatter. So $e^{it} Q$ is a blowup solution, even though it is global.
\end{remark}

Let $t_{n} \nearrow T^{+}(u)$ be a sequence of times. Making a profile decomposition, after passing to a subsequence, for all $J$,
\begin{equation}\label{2.1}
u(t_{n}) = \sum_{j = 1}^{J} g_{n}^{j} \phi^{j} + w_{n}^{J},
\end{equation}
where $g_{n}^{j}$ is the group action
\begin{equation}\label{2.1.1}
g_{n}^{j} \phi^{j} = \lambda_{n, j}^{1/2} e^{ix \xi_{n,j}} e^{i \gamma_{n,j}} \phi^{j}(\lambda_{n,j} x + x_{n,j}).
\end{equation}
Since $u$ is a minimal mass blowup solution, $\phi^{j} = 0$ for $j \geq 2$, $\| \phi^{1} \|_{L^{2}} = \| Q \|_{L^{2}}$, and $\| w_{n}^{J} \|_{L^{2}} \rightarrow 0$ as $n \rightarrow \infty$. See \cite{dodson2019defocusing} or \cite{tao2008minimal} for a detailed treatment of the profile decomposition for minimal mass blowup solutions. Thus, it will be convenient to drop the $j$ notation and simply write,
\begin{equation}\label{2.1.2}
u(t_{n}) = g_{n} \phi + w_{n}.
\end{equation}

Let $v$ be the solution to $(\ref{1.1})$ with initial data $\phi$, and let $I$ be the maximal interval of existence of $v$. Since
\begin{equation}\label{2.1.3}
\lim_{n \rightarrow \infty} \| u \|_{L_{t,x}^{6}((T^{-}(u), t_{n}) \times \mathbb{R})} = \infty, \qquad \text{and} \qquad \| u \|_{L_{t,x}^{6}((t_{n}, T^{+}(u)) \times \mathbb{R})} = \infty \qquad \forall n,
\end{equation}
\begin{equation}\label{2.2}
\| v \|_{L_{t,x}^{6}([0, \sup(I)) \times \mathbb{R})} = \| v \|_{L_{t,x}^{6}((\inf(I), 0] \times \mathbb{R})} = \infty.
\end{equation}
\begin{remark}
Equation $(\ref{2.1.3})$ is also the reason that it was unnecessary to allow for the possibility of terms like $[e^{it_{n}^{j} \Delta} \phi^{j}]$ in $(\ref{2.1})$ in place of $\phi^{j}$, where $t_{n}^{j} \rightarrow \pm \infty$.
\end{remark}

\begin{theorem}\label{t2.1}
To prove Theorem $\ref{t1.2}$, it suffices to prove that there exists a sequence $s_{m} \nearrow \sup(I)$, $s_{m} \geq 0$, such that
\begin{equation}\label{2.2.1}
g(s_{m}) v(s_{m}) \rightarrow Q, \qquad \text{in} \qquad L^{2},
\end{equation}
where $g(s_{m})$ is in the form of $(\ref{2.1.1})$.
\end{theorem}
\begin{proof}
For any $m$ let $s_{m} \in I$ be such that
\begin{equation}\label{2.4}
\| g(s_{m}) v(s_{m}) - Q \|_{L^{2}} \leq 2^{-m}.
\end{equation}
Next, observe that $(\ref{2.1})$ implies
\begin{equation}\label{2.3}
e^{i \xi_{n} x} e^{i \gamma_{n}} \lambda_{n}^{1/2} u(t_{n}, \lambda_{n} x + x_{n}) \rightarrow \phi, \qquad \text{in} \qquad L^{2},
\end{equation}
and by perturbation theory, for a fixed $m$, for $n$ sufficiently large,
\begin{equation}\label{2.5}
\aligned
\| \lambda_{n}^{1/2} e^{-i \xi_{n}^{2} s_{m}} e^{i \xi_{n} x} e^{i \gamma_{n}} u(t_{n} + \lambda_{n}^{2} s_{m}, \lambda_{n} x + x_{n} - 2 \xi_{n} \lambda_{n} s_{m}) - v(s_{m}) \|_{L^{2}} \\ \leq C(s_{m}) \| e^{i \xi_{n} x} e^{i \gamma_{n}} \lambda_{n}^{1/2} u(t_{n}, \lambda_{n} x + x_{n}) - \phi \|_{L^{2}}.
\endaligned
\end{equation}
Therefore, by $(\ref{2.4})$, $(\ref{2.5})$, and the triangle inequality,
\begin{equation}\label{2.6}
\aligned
\| g(s_{m})(\lambda_{n}^{1/2} e^{-i \xi_{n}^{2} s_{m}} e^{i \xi_{n} x} e^{i \gamma_{n}} u(t_{n} + \lambda_{n}^{2} s_{m}, \lambda_{n} x + x_{n} - 2 \xi_{n} \lambda_{n} s_{m})) - Q \|_{L^{2}} \\ \leq C(s_{m}) \| e^{i \xi_{n} x} e^{i \gamma_{n}} \lambda_{n}^{1/2} u(t_{n}, \lambda_{n} x + x_{n}) - \phi \|_{L^{2}} + 2^{-m}.
\endaligned
\end{equation}
Since $g(s_{m})$ is also of the form $(\ref{2.1.1})$, there exists a group action $g_{n,m}$ of the form $(\ref{2.1.1}$ such that
\begin{equation}\label{2.6.1}
g(s_{m})(\lambda_{n}^{1/2} e^{-i \xi_{n}^{2} s_{m}} e^{i \xi_{n} x} e^{i \gamma_{n}} u(t_{n} + \lambda_{n}^{2} s_{m}, \lambda_{n} x + x_{n} - 2 \xi_{n} \lambda_{n} s_{m})) = g_{n,m} u(t_{n} + \lambda_{n}^{2} s_{m}, x).
\end{equation}
Equation $(\ref{2.6})$ implies
\begin{equation}\label{2.7}
\lim_{m, n \rightarrow \infty} \| g_{n,m} u(t_{n} + \lambda_{n}^{2} s_{m}, x)  - Q \|_{L^{2}} = 0.
\end{equation}
Since $t_{n} \nearrow T^{+}(u)$ and $s_{m} \geq 0$, $t_{n} + \lambda_{n}^{2} s_{m} \nearrow T^{+}(u)$.
\end{proof}

Now then, since we know that $v(s)$ blows up in both time directions, $(\ref{2.2})$ holds, and $\| v \|_{L^{2}} = \| Q \|_{L^{2}}$, Theorem $1.13$ of \cite{tao2008minimal} implies that $v$ is almost periodic. That is, for all $s \in I$, there exist $\lambda(s) > 0$, $\xi(s) \in \mathbb{R}$, $x(s) \in \mathbb{R}$, and $\gamma(s) \in \mathbb{R}$ such that
\begin{equation}\label{2.8}
\lambda(s)^{-1/2} e^{ix \xi(s)} e^{i \gamma(s)} v(s, \frac{x - x(s)}{\lambda(s)}) \in K,
\end{equation}
where $K$ is a fixed precompact subset of $L^{2}$. It only remains to prove sequential convergence to $Q$ for this solution $v$.
\begin{theorem}\label{t2.2}
There exists a sequence $s_{m} \nearrow \sup(I)$ and a sequence of group actions $g(s_{m})$ of the form $(\ref{2.1.1})$ such that
\begin{equation}\label{2.9}
\| g(s_{m}) v(s_{m}) - Q \|_{L^{2}} \rightarrow 0.
\end{equation}
\end{theorem}
The proof of this fact will occupy the next two sections.
\begin{remark}
In order for notation to align with notation in prior works, such as \cite{dodson2015global}, it will be convenient to relabel so that $v$ is now denoted $u$, and $s$ now denoted $t$.
\end{remark}

\section{Proof of Theorem $\ref{t2.2}$ when $\lambda(t) = 1$}
When $\lambda(t) = 1$, the solution $u$ is global in both time directions, $I = \mathbb{R}$. As in \cite{dodson2015global} use the interaction Morawetz quantity
\begin{equation}\label{3.1}
M(t) = \int \int |Iu(t, y)|^{2} Im[\bar{Iu} Iu_{x}] \psi(x - y) dx dy,
\end{equation}
where $I$ is the Fourier truncation operator $P_{\leq T}$, where $T = 2^{k}$, where $k \in \mathbb{Z}_{\geq 0}$.
Here,
\begin{equation}\label{3.2}
\psi(x) = \int_{0}^{x} \phi(s) ds,
\end{equation}
where $\phi(s)$ is an even function given by
\begin{equation}\label{3.3}
\phi(x - y) = \frac{1}{R} \int \chi^{2}(\frac{x - y - s}{R}) \chi^{2}(\frac{s}{R}) ds = \frac{1}{R} \int \chi^{2}(\frac{x - s}{R}) \chi^{2}(\frac{s - y}{R}) ds = \frac{1}{R} \int \chi^{2}(\frac{x - s}{R}) \chi^{2}(\frac{y - s}{R}) ds,
\end{equation}
where $\chi$ is a smooth, compactly supported, even function, $\chi(x) = 1$ for $|x| \leq 1$ and $\chi(x)$ is supported on $|x| \leq 2$, and $\chi(x)$ is decreasing on the set $1 \leq x \leq 2$. $R$ is a large, fixed constant that will be allowed to go to infinity as $T \rightarrow \infty$.\medskip

By direct computation,
\begin{equation}\label{3.4}
\aligned
\frac{d}{dt} M(t) = -2 \int \int Im[\bar{Iu} Iu_{y}] Im[\bar{Iu} Iu_{x}] \phi(x - y) dx dy \\
+ \frac{1}{2} \int \int |Iu(t,y)|^{2} |Iu(t,y)|^{2} \phi''(x - y) dx dy \\
+ 2 \int \int |Iu(t,y)|^{2} |Iu_{x}(t,x)|^{2} \phi(x - y) dx dy \\
- \frac{2}{3} \int \int |Iu(t,y)|^{2} |Iu(t,x)|^{6} \phi(x - y) dx dy + \mathcal E,
\endaligned
\end{equation}
where $\mathcal E$ are the error terms arising from $\mathcal N$,
\begin{equation}\label{3.4.1}
i Iu_{t} + I u_{xx} + F(Iu) = F(Iu) - I F(u) = \mathcal N.
\end{equation}
It is known from Theorem $1.13$ of \cite{dodson2015global} that
\begin{equation}\label{3.4.2}
\int_{0}^{T} \mathcal N dt \lesssim R o(T),
\end{equation}
and
\begin{equation}\label{3.4.3}
\sup_{t \in [0, T]} |M(t)| \lesssim R o(T).
\end{equation}

By direct computation,
\begin{equation}\label{3.5}
\phi(x) = \frac{1}{R} \int \chi^{2}(\frac{x - s}{R}) \chi^{2}(\frac{s}{R}) ds \sim 1,
\end{equation}
for $|x| \leq R$, and $\phi(x)$ is supported on the set $|x| \leq 4R$. Finally, $\phi(x)$ is decreasing when $x \geq 0$. Therefore, $(\ref{3.2})$ implies that
\begin{equation}\label{3.6}
|\psi(x)| \lesssim R.
\end{equation}
Also, by direct computation,
\begin{equation}\label{3.7}
\phi''(x) = \frac{1}{R} \int \partial_{xx}(\chi^{2}(\frac{x - s}{R})) \chi^{2}(\frac{s}{R}) ds \lesssim \frac{1}{R^{2}}.
\end{equation}
Therefore,
\begin{equation}\label{3.8}
 \frac{1}{2} \int \int |Iu(t,y)|^{2} |Iu(t,y)|^{2} \phi''(x - y) dx dy \lesssim \frac{1}{R^{2}} \| u \|_{L^{2}}^{4}.
 \end{equation}

Next, decompose
\begin{equation}\label{3.9}
\aligned
-\int \int Im[\bar{Iu} Iu_{y}] Im[\bar{Iu} Iu_{x}] \phi(x - y) dx dy + \int \int |Iu(t,y)|^{2} |Iu_{x}(t,x)|^{2} \phi(x - y) dx dy \\
= \frac{1}{R} \int (\int \chi^{2}(\frac{y - s}{R}) Im[\bar{Iu} Iu_{y}])(\int \chi^{2}(\frac{x - s}{R}) Im[\bar{Iu} Iu_{x}]) ds \\ + \frac{1}{R} \int (\int \chi^{2}(\frac{y - s}{R}) |Iu(t,y)|^{2} dy)(\int \chi^{2}(\frac{x - s}{R}) |Iu_{x}(t,x)|^{2} dx) ds.
\endaligned
\end{equation}
Fix $s \in \mathbb{R}$. For any $\xi \in \mathbb{R}$,
\begin{equation}\label{3.10}
\int \chi^{2}(\frac{y - s}{R}) Im[\overline{e^{iy \xi} Iu} \partial_{y}(e^{iy \xi} Iu)] dy = \int \chi^{2}(\frac{y - s}{R}) Im[\bar{Iu} Iu_{y}] dy + \xi \int \chi^{2}(\frac{y - s}{R}) |Iu(t,y)|^{2} dy,
\end{equation}
and
\begin{equation}\label{3.11}
\int \chi^{2}(\frac{x - s}{R}) |\partial_{x}(e^{ix \xi} Iu)|^{2} dx = \xi^{2} \int \chi^{2}(\frac{x - s}{R}) |Iu|^{2} dx + 2 \xi \int \chi^{2}(\frac{x - s}{R}) Im[\bar{Iu} Iu_{x}] dx + \int \chi^{2}(\frac{x - s}{R}) |Iu_{x}|^{2} dx.
\end{equation}
Therefore, $(\ref{3.9})$ is invariant under the Galilean transformation $Iu \mapsto e^{ix \xi(s)} Iu$. This is not surprising since $(\ref{3.1})$ is invariant under the Galilean transformation $Iu \mapsto e^{ix \xi(s)} Iu$. Indeed, under the mapping $e^{ix \xi(s)} Iu$, since $\psi(x - y)$ is an odd function of $x - y$,
\begin{equation}
\aligned
M(t) \mapsto \int \int |Iu(t,y)|^{2} Im[\bar{Iu} Iu_{x}] \psi(x - y) dx dy + \xi(t) \int \int |Iu(t,y)|^{2} |Iu(t,x)|^{2} dx dy \\
= \int \int |Iu(t,y)|^{2} Im[\bar{Iu} Iu_{x}] \psi(x - y) dx dy.
\endaligned
\end{equation}

It is therefore convenient to choose $\xi(s)$ such that $(\ref{3.10}) = 0$. For notational convenience, let 
\begin{equation}\label{3.12}
v_{s} = e^{ix \xi(s)} Iu.
\end{equation}
Then by the fundamental theorem of calculus and $(\ref{3.4.2})$--$(\ref{3.12})$, if $R \nearrow \infty$ as $T \nearrow \infty$,
\begin{equation}\label{3.13}
\aligned
2 \int_{0}^{T} \frac{1}{R} \int (\int \chi^{2}(\frac{y - s}{R}) |v_{s}(t,y)|^{2}) (\int \chi^{2}(\frac{x - s}{R}) |\partial_{x}(v_{s})(t,x)|^{2} dx) ds dt \\
- \frac{2}{3} \int_{0}^{T} \frac{1}{R} \int (\int \chi^{2}(\frac{y - s}{R}) |v_{s}(t,y)|^{2}) (\int \chi^{2}(\frac{x - s}{R}) |v_{s}(t,x)|^{6} dx) ds dt \lesssim R o(T).
\endaligned
\end{equation}
By the Arzela--Ascoli theorem and $(\ref{2.8})$, for any $\eta > 0$, there exists $C(\eta) < \infty$ such that
\begin{equation}\label{3.13.1}
\int_{|x - x(t)| \geq \frac{C(\eta)}{\lambda(t)}} |u(t,x)|^{2} dx < \eta^{2}.
\end{equation}
By H{\"o}lder's inequality and $\lambda(t) = 1$,
\begin{equation}\label{3.14}
\frac{1}{6} \int_{|x - x(t)| \geq C(\eta)} \chi^{2}(\frac{x - s}{R}) |v_{s}|^{6} dx \lesssim \| \chi v_{s}^{2} \|_{L^{\infty}(|x - x(t)| \geq C(\eta))}^{2} (\int_{|x - x(t)| \geq C(\eta)} |v_{s}|^{2} dx).
\end{equation}

We can estimate $\| \chi v_{s}^{2} \|_{L^{\infty}(|x - x(t)| \geq C(\eta))}^{2}$ using an idea from \cite{merle2001existence}. By the fundamental theorem of calculus, for $|x - x(t)| \geq C(\eta)$,
\begin{equation}\label{3.15}
\aligned
|\chi v_{s}^{2}| \leq (\int 2 \chi |v| |v_{x}| dx + \frac{1}{R} \int |\chi'| |v|^{2} dx) \\ \lesssim (\int \chi^{2}(\frac{x - s}{R}) |\partial_{x}(v_{s})|^{2} dx)^{1/2} (\int_{|x - x(t)| \geq C(\eta)} |v|^{2})^{1/2} + \frac{1}{R} (\int_{|x - x(t)| \geq C(\eta)} |v|^{2}).
\endaligned
\end{equation}
Therefore,
\begin{equation}\label{3.16}
\aligned
(\ref{3.14}) \lesssim (\int \chi^{2}(\frac{x - s}{R}) |\partial_{x}(v_{s})|^{2} dx)(\int_{|x - x(t)| \geq C(\eta)} |v|^{2} dx)^{2} + \frac{1}{R} (\int_{|x - x(t)| \geq C(\eta)} |v|^{2})^{3} \\
\lesssim \eta^{4} (\int \chi^{2}(\frac{x - s}{R}) |\partial_{x}(v_{s})|^{2} dx) + \frac{1}{R} \eta^{6}.
\endaligned
\end{equation}
Therefore,
\begin{equation}\label{3.17}
\aligned
2 \int_{0}^{T} \frac{1}{R} \int (\int \chi^{2}(\frac{y - s}{R}) |v_{s}(t,y)|^{2}) (\int \chi^{2}(\frac{x - s}{R}) |\partial_{x}(v_{s})(t,x)|^{2} dx) ds dt \\
- \frac{2}{3} \int_{0}^{T} \frac{1}{R} \int (\int \chi^{2}(\frac{y - s}{R}) |v_{s}(t,y)|^{2}) (\int_{|x - x(t)| \leq C(\eta)} \chi^{2}(\frac{x - s}{R}) |v_{s}(t,x)|^{6} dx) ds dt \\ \lesssim R o(T) + \frac{\eta^{6}}{R} T + \frac{\eta^{4}}{R} \int_{0}^{T} \int (\int \chi^{2}(\frac{y - s}{R}) |v_{s}|^{2} dy) (\int \chi^{2}(\frac{x - s}{R}) |\partial_{x}(v_{s})|^{2} dx) ds dt.
\endaligned
\end{equation}
When $\eta > 0$ is sufficiently small,
\begin{equation}\label{3.26}
 \eta^{4} \int_{0}^{T} \int (\int \chi^{2}(\frac{y - s}{R}) |v_{s}|^{2} dy) (\int \chi^{2}(\frac{x - s}{R}) |\partial_{x}(v_{s})|^{2} dx) ds dt
 \end{equation}
 can be absorbed into the left hand side of $(\ref{3.17})$, proving that
 \begin{equation}\label{3.27}
\aligned
2 \int_{0}^{T} \frac{1}{R} \int (\int \chi^{2}(\frac{y - s}{R}) |v_{s}(t,y)|^{2}) (\int \chi^{2}(\frac{x - s}{R}) |\partial_{x}(v_{s})(t,x)|^{2} dx) ds dt \\
- \frac{2}{3} \int_{0}^{T} \frac{1}{R} \int (\int \chi^{2}(\frac{y - s}{R}) |v_{s}(t,y)|^{2}) (\int_{|x - x(t)| \leq C(\eta)} \chi^{2}(\frac{x - s}{R}) |v_{s}(t,x)|^{6} dx) ds dt \\ \lesssim R o(T) + \frac{\eta^{6}}{R} T + \frac{ \eta^{4}}{R} \int_{0}^{T} \int (\int \chi^{2}(\frac{y - s}{R}) |v_{s}(t,y)|^{2} dy) (\int_{|x - x(t)| \leq C(\eta)} \chi^{2}(\frac{x - s}{R}) |v_{s}(t,x)|^{6} dx) ds dt.
\endaligned
\end{equation}

Now choose $x(t) - 2 C(\eta) \leq x_{\ast} \leq x(t) + 2 C(\eta)$ such that
\begin{equation}\label{3.17.1}
\chi(\frac{x_{\ast} - s}{R}) = \inf_{x(t) - 2 C(\eta) \leq x \leq x(t) + 2 C(t)} \chi(\frac{x - s}{R}).
\end{equation}
By the fundamental theorem of calculus, when $x(t) - C(\eta) \leq x \leq x(t) + C(\eta)$,
\begin{equation}\label{3.18}
\chi^{2}(\frac{x - s}{R}) = \chi^{2}(\frac{x_{\ast} - s}{R}) + \frac{2}{R} \int_{x_{\ast}}^{x} \chi'(\frac{r - s}{R}) \chi(\frac{r - s}{R}) dr.
\end{equation}
When $|x - x(t)| \leq C(\eta)$, 
\begin{equation}\label{3.19}
\frac{2}{R} \int_{x_{\ast}}^{x} \chi'(\frac{r - s}{R}) \chi(\frac{r - s}{R}) dr \lesssim \frac{C(\eta)}{R},\end{equation}
so
\begin{equation}\label{3.20}
\aligned
\frac{1}{R} \int (\int \chi^{2}(\frac{y - s}{R}) |v_{s}(t,y)|^{2}) (\int_{|x - x(t)| \leq C(\eta)} \chi^{2}(\frac{x - s}{R}) |v_{s}(t,x)|^{6} dx) ds \\ \leq \frac{1}{R} \int (\int \chi^{2}(\frac{y - s}{R}) |v_{s}(t,y)|^{2}) (\int_{|x - x(t)| \leq C(\eta)} \chi^{2}(\frac{x_{\ast} - s}{R}) |v_{s}(t,x)|^{6} dx) ds \\
+ \frac{C(\eta)}{R^{2}} \int (\int \chi^{2}(\frac{y - s}{R}) |v_{s}(t,y)|^{2}) (\int_{|x - x(t)| \leq C(\eta)} |v_{s}(t,x)|^{6} dx) ds \\
=  \frac{1}{R} \int (\int \chi^{2}(\frac{y - s}{R}) |v_{s}(t,y)|^{2}) (\int_{|x - x(t)| \leq C(\eta)} \chi^{2}(\frac{x_{\ast} - s}{R}) |v_{s}(t,x)|^{6} dx) ds + O(\frac{C(\eta)}{R} \| v \|_{L^{2}}^{2} \| v \|_{L^{6}}^{6}).
\endaligned
\end{equation}

Plugging $(\ref{3.20})$ in to $(\ref{3.27})$,
\begin{equation}\label{3.21}
\aligned
2 \int_{0}^{T} \frac{1}{R} \int (\int \chi^{2}(\frac{y - s}{R}) |v_{s}(t,y)|^{2}) (\int \chi^{2}(\frac{x - s}{R}) |\partial_{x}(v_{s})(t,x)|^{2} dx) ds dt \\
- \frac{2}{3} \int_{0}^{T} \frac{1}{R} \int (\int \chi^{2}(\frac{y - s}{R}) |v_{s}(t,y)|^{2}) (\int_{|x - x(t)| \leq C(\eta)} \chi^{2}(\frac{x_{\ast} - s}{R}) |v_{s}(t,x)|^{6} dx) ds dt \\ \lesssim R o(T) + \frac{\eta^{6}}{R} T + \frac{C(\eta)}{R} \| u \|_{L_{t}^{\infty} L_{x}^{2}}^{2} \| u \|_{L_{t,x}^{6}}^{6} \\ + \frac{\eta^{4}}{R} \int_{0}^{T} \int (\int \chi^{2}(\frac{y - s}{R}) |v_{s}|^{2} dy) (\int_{|x - x(t)| \leq C(\eta)} \chi^{2}(\frac{x^{\ast} - s}{R}) |v_{s}|^{6} dx) ds dt .
\endaligned
\end{equation}
Since $\chi(\frac{x^{\ast} - 1}{R}) \leq 1$, 
\begin{equation}
\frac{\eta^{4}}{R} \int_{0}^{T} \int (\int \chi^{2}(\frac{y - s}{R}) |v_{s}|^{2} dy) (\int_{|x - x(t)| \leq C(\eta)} \chi^{2}(\frac{x^{\ast} - s}{R}) |v_{s}|^{6} dx) ds dt \lesssim \eta^{4} \| u \|_{L_{t}^{\infty} L_{x}^{2}}^{2} \| u \|_{L_{t,x}^{6}}^{6}.
\end{equation}
By definition of $x_{\ast}$ and $\chi$,
\begin{equation}\label{3.22}
\aligned
2 \int_{0}^{T} \frac{1}{R} \int (\int \chi^{2}(\frac{y - s}{R}) |v_{s}(t,y)|^{2}) (\int \chi^{2}(\frac{x - s}{R}) |\partial_{x}(v_{s})(t,x)|^{2} dx) ds dt \\
- \frac{2}{3} \int_{0}^{T} \frac{1}{R} \int (\int \chi^{2}(\frac{y - s}{R}) |v_{s}(t,y)|^{2}) (\int_{|x - x(t)| \leq C(\eta)} \chi^{2}(\frac{x_{\ast} - s}{R}) |v_{s}(t,x)|^{6} dx) ds dt \\
\geq 2 \int_{0}^{T} \frac{1}{R} \int (\int \chi^{2}(\frac{y - s}{R}) |v_{s}(t,y)|^{2}) ( \chi^{2}(\frac{x_{\ast} - s}{R}) \int \chi^{2}(\frac{x - x(t)}{R}) |\partial_{x}(v_{s})(t,x)|^{2} dx) ds dt \\
- \frac{2}{3} \int_{0}^{T} \frac{1}{R} \int (\int \chi^{2}(\frac{y - s}{R}) |v_{s}(t,y)|^{2}) (\chi^{2}(\frac{x_{\ast} - s}{R})\int \chi^{6}(\frac{x - x(t)}{R}) |v_{s}(t,x)|^{6} dx) ds dt.
\endaligned
\end{equation}
Integrating by parts,
\begin{equation}\label{3.23}
\int \chi^{2}(\frac{x - x(t)}{R}) |\partial_{x}(v_{s})|^{2} dx = \int |\partial_{x}(\chi(\frac{x - x(t)}{R}) v_{s})|^{2} dx + \frac{1}{R^{2}} \int \chi''(\frac{x - x(t)}{R}) \chi(\frac{x - x(t)}{R}) |v_{s}|^{2} dx.
\end{equation}
Therefore,
\begin{equation}\label{3.24}
\aligned
2 \int_{0}^{T} \frac{1}{R} \int (\int \chi^{2}(\frac{y - s}{R}) |v_{s}(t,y)|^{2}) ( \chi^{2}(\frac{x_{\ast} - s}{R}) \int \chi^{2}(\frac{x - x(t)}{R}) |\partial_{x}(v_{s})(t,x)|^{2} dx) ds dt \\
- \frac{2}{3} \int_{0}^{T} \frac{1}{R} \int (\int \chi^{2}(\frac{y - s}{R}) |v_{s}(t,y)|^{2}) (\chi^{2}(\frac{x_{\ast} - s}{R})\int \chi^{6}(\frac{x - x(t)}{R}) |v_{s}(t,x)|^{6} dx) ds dt \\
= 4 \int_{0}^{T} \frac{1}{R} \int (\int \chi^{2}(\frac{y - s}{R}) |v_{s}(t, y)|^{2} dy) \chi^{2}(\frac{x_{\ast} - s}{R}) E(\chi^{2}(\frac{x - x(t)}{R}) v) ds dt + \frac{T}{R^{2}} \| u \|_{L_{t}^{\infty} L_{x}^{2}}^{4}.
\endaligned
\end{equation}
Here $E$ is the energy given by $(\ref{1.7})$. Therefore, we have finally proved
\begin{equation}\label{3.25}
\aligned
4 \int_{0}^{T} \frac{1}{R} \int (\int \chi^{2}(\frac{y - s}{R}) |v_{s}(t, y)|^{2} dy) \chi^{2}(\frac{x_{\ast} - s}{R}) E(\chi^{2}(\frac{x - x(t)}{R}) v) ds dt \\ \lesssim R o(T) + \frac{\eta^{6}}{R} T + \eta^{4} \| u \|_{L_{t}^{\infty} L_{x}^{2}}^{2} \| u \|_{L_{t,x}^{6}}^{6} + \frac{C(\eta)}{R} \| u \|_{L_{t}^{\infty} L_{x}^{2}}^{2} \| u \|_{L_{t,x}^{6}}^{6}.
\endaligned
\end{equation}

Now by Strichartz estimates and $\lambda(t) = 1$, $\| u \|_{L_{t,x}^{6}([0, T] \times \mathbb{R})} \sim T$, so choosing $R \nearrow \infty$ perhaps very slowly as $T \nearrow \infty$, and then $\eta \searrow 0$ sufficiently slowly, the right hand side of $(\ref{3.25})$ is bounded by $o(T)$.\medskip

On the other hand, when $|s - x(t)| \leq \frac{R}{2}$, $\chi(\frac{x_{\ast} - s}{R}) = 1$ and
\begin{equation}\label{3.28}
(\int \chi^{2}(\frac{y - s}{R}) |v_{s}(t, y)|^{2} dy) \geq \frac{1}{2} \| v \|_{L^{2}}^{2}.
\end{equation}
Therefore, the left hand side of $(\ref{3.25})$ is bounded below by
\begin{equation}\label{3.29}
\| u_{0} \|_{L^{2}}^{2} \int_{0}^{T} \frac{1}{R} \int_{|s - x(t)| \leq \frac{R}{2}} E(\chi(\frac{x - x(t)}{R}) v_{s}) ds dt \lesssim o(T)
\end{equation}
Thus, taking a sequence $T_{n} \nearrow \infty$, $R_{n} \nearrow \infty$, $\eta_{n} \searrow 0$, there exists a sequence of times $t_{n} \in [\frac{T_{n}}{2}, T_{n}]$, $|s_{n} - x(t_{n})| \leq \frac{R_{n}}{2}$ such that
\begin{equation}\label{3.30}
E(\chi(\frac{x - s_{n}}{R_{n}}) e^{ix \xi(s_{n})} e^{i \gamma(s_{n})} P_{\leq T_{n}} u(t_{n}, x)) \rightarrow 0,
\end{equation}
\begin{equation}\label{3.31}
(1 - \chi(\frac{x - s_{n}}{R_{n}})) e^{ix \xi(s_{n})} e^{i \gamma(s_{n})} P_{\leq T_{n}} u(t_{n}, x) \rightarrow 0, \qquad \text{in} \qquad L^{2},
\end{equation}
\begin{equation}\label{3.31.1}
 (1 - P_{\geq T_{n}}) u(t_{n}, x) \rightarrow 0, \qquad \text{in} \qquad L^{2},
 \end{equation}
and
\begin{equation}\label{3.32}
\| \chi(\frac{x - s_{n}}{R_{n}}) e^{ix \xi(s_{n})} e^{i \gamma(s_{n})} P_{\leq T_{n}} u(t_{n}, x) \|_{L^{6}} \sim 1.
\end{equation}

Now by the almost periodicity of $u$, $(\ref{2.8})$, after passing to a subsequence, there exists $u_{0} \in H^{1}$ such that
\begin{equation}\label{3.33}
\chi(\frac{x + x(t_{n}) - s_{n}}{R_{n}}) e^{ix \xi(s_{n})} e^{i x(t_{n}) \xi(s_{n})} e^{i \gamma(s_{n})} P_{\leq T_{n}} u(t_{n}, x + x(t_{n})) \rightharpoonup u_{0},
\end{equation}
weakly in $H^{1}$, and
\begin{equation}\label{3.34}
\chi(\frac{x + x(t_{n}) - s_{n}}{R_{n}}) e^{ix \xi(s_{n})} e^{i x(t_{n}) \xi(s_{n})} P_{\leq T_{n}} u(t_{n}, x + x(t_{n})) \rightarrow u_{0},
\end{equation}
strongly in in $L^{2} \cap L^{6}$. Also, by $(\ref{3.30})$, $(\ref{3.31})$, and $(\ref{3.31.1})$, $\| u_{0} \|_{L^{2}} = \| Q \|_{L^{2}}$, $E(u_{0}) \leq 0$, and by the Gagliardo-Nirenberg inequality, $E(u_{0}) = 0$. Therefore,
\begin{equation}\label{3.36}
u_{0} = \lambda^{1/2} Q(\lambda(x - x_{0})),
\end{equation}
for some $\lambda \sim 1$ and $|x_{0}| \lesssim 1$. This proves Theorem $\ref{t2.2}$ when $\lambda(t) = 1$. $\Box$

\section{Proof of Theorem $\ref{t2.2}$ for a general $\lambda(t)$}
Now suppose that $\lambda(t)$ is free to vary. Recall that $|\lambda'(t)| \lesssim \lambda(t)^{3}$. In this case, 
\begin{equation}\label{4.1}
\lambda(t) : I \rightarrow (0, \infty),
\end{equation}
where $I$ is the maximal interval of existence of an almost periodic solution to $(\ref{1.1})$.
\begin{theorem}\label{t4.1}
Suppose $T_{n} \in I$, $T_{n} \rightarrow \sup(I)$ is a sequence of times in $I$. Then
\begin{equation}\label{4.2}
\lim_{T_{n} \rightarrow \sup(I)} \frac{1}{\sup_{t \in [0, T_{n}]} \lambda(t)} \cdot \int_{0}^{T_{n}} \lambda(t)^{3} dt = +\infty.
\end{equation}
\end{theorem}
\noindent \emph{Proof:} Suppose that this were not true, that is, there exists a constant $C_{0} < \infty$ and a sequence $T_{n} \rightarrow \sup(I)$ such that for all $n \in \mathbb{Z}_{\geq 0}$,
\begin{equation}\label{4.3}
\frac{1}{\sup_{t \in [0, T_{n}]} \lambda(t)} \int_{0}^{T_{n}} \lambda(t)^{3} dt \leq C_{0}.
\end{equation}
This would correspond to the rapid cascade scenario in \cite{dodson2015global}, \cite{dodson2016global}, \cite{fan20182}. In those papers $N(t)$ was used instead of $\lambda(t)$. As in those papers, $\lambda(t)$ can be chosen to be continuous, so for each $T_{n}$ choose $t_{n} \in [0, T_{n}]$ such that
\begin{equation}\label{4.4}
\lambda(t_{n}) = \sup_{t \in [0, T_{n}]} \lambda(t).
\end{equation}

Since $I$ is the maximal interval of existence of $u$,
\begin{equation}\label{4.5}
\lim_{n \rightarrow \infty} \| u \|_{L_{t,x}^{6}([0, T_{n}] \times \mathbb{R})} = \infty.
\end{equation}
By the almost periodicity property of $u$ and $(\ref{2.8})$, there exist $x(t_{n})$, $\xi(t_{n})$, and $\gamma(t_{n})$ such that if
\begin{equation}\label{4.7}
e^{i \gamma(t_{n})} \lambda(t_{n})^{1/2} e^{ix \xi(t_{n})} e^{i \gamma(t_{n})} u(t_{n}, \lambda(t_{n}) x + x(t_{n})) = v_{n}(x),
\end{equation}
then $v_{n}$ converges to some $u_{0}$ in $L^{2}(\mathbb{R})$, and $u_{0}$ is the initial data for a solution $u$ to $(\ref{1.1})$ that blows up in both time directions, $\lambda(t) \leq 1$ for all $t \leq 0$, and
\begin{equation}\label{4.8}
\int_{-\infty}^{0} \lambda(t)^{3} dt \leq C_{0}.
\end{equation}
Following the proof of Theorem $5.1$ in \cite{dodson2016global},
\begin{equation}\label{4.9}
\| u \|_{L_{t}^{\infty} \dot{H}^{s}((-\infty, 0] \times \mathbb{R})} \lesssim_{s} C_{0}^{s},
\end{equation}
for any $0 \leq s < 5$. Combining $(\ref{4.9})$ with $(\ref{4.8})$ and $|\lambda'(t)| \lesssim \lambda(t)^{3}$ implies
\begin{equation}\label{4.10}
\lim_{t \searrow -\infty} \lambda(t) = 0.
\end{equation}
Also, since
\begin{equation}\label{4.11}
|\xi'(t)| \lesssim \lambda(t)^{3},
\end{equation}
Equation $(\ref{4.8})$ implies that $\xi(t)$ converges to some $\xi_{-} \in \mathbb{R}$ as $t \searrow -\infty$. Make a Galilean transformation so that $\xi_{-} = 0$. Then, by interpolation, $(\ref{4.9})$ and $(\ref{4.10})$ imply
\begin{equation}\label{4.12}
\lim_{t \searrow -\infty} E(u(t)) = 0.
\end{equation}

Therefore, by conservation of energy, and convergence in $L^{2}$ of $(\ref{4.7})$,
\begin{equation}\label{4.13}
E(u_{0}) = 0, \qquad \text{and} \qquad \| u_{0} \|_{L^{2}} = \| Q \|_{L^{2}}.
\end{equation}
Therefore, by the Gagliardo-Nirenberg theorem,
\begin{equation}\label{4.14}
u_{0} = \lambda^{1/2} Q(\lambda(x - x_{0})), \qquad 0 < \lambda < \infty, \qquad x_{0} \in \mathbb{R},
\end{equation}
and $Q$ is the solution to the elliptic partial differential equation
\begin{equation}\label{4.15}
Q_{xx} + |Q|^{4} Q = Q.
\end{equation}
However, assuming without loss of generality that $x_{0} = 0$ and $\lambda = 1$, the solution to $(\ref{1.1})$ is given by
\begin{equation}\label{4.16}
u(t,x) = e^{it} Q(x), \qquad t \in \mathbb{R}.
\end{equation}
However, such a solution definitely does not satisfy $(\ref{4.3})$, which gives a contradiction. $\Box$\medskip

Therefore, consider the case when
\begin{equation}\label{4.17}
\lim_{n \rightarrow \infty} \frac{1}{\sup_{t \in [0, T_{n}]} \lambda(t)} \int_{0}^{T_{n}} \lambda(t)^{3} dt = \infty.
\end{equation}
Passing to a subsequence, suppose
\begin{equation}\label{4.18}
\frac{1}{\sup_{t \in [0, T_{n}]} \lambda(t)} \int_{0}^{T_{n}} \lambda(t)^{3} dt = 2^{2n}.
\end{equation}

Then as in \cite{dodson2015global}, replace $M(t)$ in the previous section with,
\begin{equation}\label{4.19}
M(t) = \int \int |Iu(t, y)|^{2} Im[\bar{Iu} Iu_{x}] \psi(\tilde{\lambda}(t) (x - y)) dx dy,
\end{equation}
where $\tilde{\lambda}(t)$ is given by the smoothing algorithm from \cite{dodson2015global}. Then
\begin{equation}\label{4.20}
\aligned
\frac{d}{dt} M(t) = -2 \tilde{\lambda}(t) \int \int Im[\bar{Iu} Iu_{y}] Im[\bar{Iu} Iu_{x}] \phi(\tilde{\lambda}(t) (x - y)) dx dy \\
+ \frac{1}{2} \tilde{\lambda}(t)^{3} \int \int |Iu(t,y)|^{2} |Iu(t,y)|^{2} \phi''(\tilde{\lambda}(t) (x - y)) dx dy \\
+ 2 \tilde{\lambda}(t) \int \int |Iu(t,y)|^{2} |Iu_{x}(t,x)|^{2}  \phi(\tilde{\lambda}(t) (x - y)) dx dy \\
- \frac{2}{3} \tilde{\lambda}(t) \int \int |Iu(t,y)|^{2} |Iu(t,x)|^{6} \phi(\tilde{\lambda}(t) (x - y)) dx dy + \mathcal E \\
+ \dot{\tilde{\lambda}}(t) \int \int |Iu(t, y)|^{2} Im[\bar{Iu} Iu_{x}] \phi(\tilde{\lambda}(t) (x - y)) (x - y) dx dy,
\endaligned
\end{equation}
where $I = P_{\leq 2^{2n} \cdot \sup_{t \in [0, T]} \lambda(t)}$.

Equation $(\ref{3.6})$ implies
\begin{equation}\label{4.21}
\sup_{t \in [0, T_{n}]} |M(t)| \lesssim R o(2^{2n}) \cdot \sup_{t \in [0, T]} \lambda(t).
\end{equation}
Next, since the smoothing algorithm guarantees that $\tilde{\lambda}(t) \leq \lambda(t)$, following $(\ref{3.8})$,
\begin{equation}\label{4.22}
\aligned
\int_{0}^{T_{n}} \frac{1}{2} \tilde{\lambda}(t)^{3} \int \int |u(t,y)|^{2} |u(t,y)|^{2} \phi''(\tilde{\lambda}(t) (x - y)) dx dy dt \\ \lesssim \frac{1}{R^{2}} \| u \|_{L^{2}}^{4} \cdot \int_{0}^{T_{n}} \tilde{\lambda}(t) \lambda(t)^{2} dt \lesssim R o(2^{2n}) \cdot \sup_{t \in [0, T]} \lambda(t).
\endaligned
\end{equation}
Since $\tilde{\lambda}(t) \leq \lambda(t)$, following the analysis in $(\ref{3.9})$--$(\ref{3.25})$,
\begin{equation}\label{4.27}
\aligned
\int_{0}^{T_{n}} \frac{1}{2} \tilde{\lambda}(t)^{3} \int \int |u(t,y)|^{2} |u(t,y)|^{2} \phi''(\tilde{\lambda}(t) (x - y)) dx dy dt \\
2 \int_{0}^{T_{n}} \tilde{\lambda}(t) \int \int Im[I\bar{u} Iu_{y}] Im[\bar{Iu} Iu_{x}] \phi(\tilde{\lambda}(t) (x - y)) dx dy dt \\
+ 2 \int_{0}^{T_{n}} \tilde{\lambda}(t) \int \int |Iu(t,y)|^{2} |Iu_{x}(t,x)|^{2}  \phi(\tilde{\lambda}(t) (x - y)) dx dy dt \\
- \frac{2}{3} \int_{0}^{T_{n}} \tilde{\lambda}(t) \int \int |u(t,y)|^{2} |u(t,x)|^{6} \phi(\tilde{\lambda}(t) (x - y)) dx dy dt
\endaligned
\end{equation}

\begin{equation}\label{4.28}
\aligned
= 4 \int_{0}^{T} \frac{\tilde{\lambda}(t) \lambda(t)^{2}}{R} \int (\int \chi^{2}(\frac{y - s}{R}) |v_{s,t}(t, y)|^{2} dy) \chi^{2}(\frac{x_{\ast} - s}{R}) E(\chi^{2}(\frac{x - x(t)}{R}) v_{s,t}(t,x)) ds dt \\ + R o(2^{2n}) \cdot \sup_{t \in [0, T]} \lambda(t) + O(\eta^{4} \| u \|_{L_{t}^{\infty} L_{x}^{2}}^{2} \int_{0}^{T_{n}} \tilde{\lambda}(t) \| u(t) \|_{L^{6}}^{6} dt) + O(\frac{C(\eta)}{R} \| u \|_{L_{t}^{\infty} L_{x}^{2}}^{2} \int_{0}^{T_{n}} \tilde{\lambda}(t) \| u(t) \|_{L_{x}^{6}}^{6} dt).
\endaligned
\end{equation}

\noindent \textbf{Remark:} The term $v_{s, t}$ is an abbreviation for
\begin{equation}\label{4.29}
v_{s, t} = \frac{e^{i x \xi(s)}}{\lambda(t)^{1/2}} Iu(t, \frac{x}{\lambda(t)}),
\end{equation}
where $\xi(s) \in \mathbb{R}$ is chosen such that
\begin{equation}\label{4.30}
\int \chi^{2}(\frac{\tilde{\lambda}(t) (x - s)}{R \lambda(t)}) Im[\bar{v}_{s, t} \partial_{x}(v_{s,t})] dx = 0.
\end{equation}

The error estimates can be handled in a manner similar to the previous section, see \cite{dodson2015global}. Therefore, it only remains to consider the contribution of the term in $(\ref{4.20})$ with $\tilde{\lambda}(t)$. By direct computation,
\begin{equation}\label{4.31}
\aligned
\dot{\tilde{\lambda}}(t) \int \int |Iu(t, y)|^{2} Im[\bar{Iu} Iu_{x}] \phi(\tilde{\lambda}(t) (x - y)) (x - y) dx dy  \\
= \frac{\dot{\tilde{\lambda}}(t)}{R \tilde{\lambda}(t)} \int (\int \chi^{2}(\frac{\tilde{\lambda}(t) y - s}{R}) |Iu(t,y)|^{2} dy)(\int \chi^{2}(\frac{\tilde{\lambda}(t) x - s}{R}) Im[\bar{Iu} Iu_{x}] (x \tilde{\lambda}(t) - s) dx) ds \\
- \frac{\dot{\tilde{\lambda}}(t)}{R \tilde{\lambda}(t)} \int (\int \chi^{2}(\frac{\tilde{\lambda}(t) y - s}{R}) (y \tilde{\lambda}(t) - s) |Iu(t,y)|^{2} dy)(\int \chi^{2}(\frac{\tilde{\lambda}(t) x - s}{R}) Im[\bar{Iu} Iu_{x}] dx) ds.
\endaligned
\end{equation}
Now rescale,
\begin{equation}\label{4.32}
\aligned
= \frac{\dot{\tilde{\lambda}}(t)}{R \tilde{\lambda}(t)} \lambda(t) \int (\int \chi^{2}(\frac{\tilde{\lambda}(t) y - \lambda(t) s}{R \lambda(t)}) |\frac{1}{\lambda(t)^{1/2}} Iu(t,\frac{y}{\lambda(t)})|^{2} dy) \\
\times (\int \chi^{2}(\frac{\tilde{\lambda}(t) x - \lambda(t) s}{R \lambda(t)}) Im[\frac{1}{\lambda(t)^{1/2}} \bar{Iu}(t, \frac{x}{\lambda(t)}) \partial_{x}(\frac{1}{\lambda(t)^{1/2}} Iu(t, \frac{x}{\lambda(t)})] (\frac{x \tilde{\lambda}(t) - s \lambda(t)}{\lambda(t)}) dx) ds \\
- \frac{\dot{\tilde{\lambda}}(t)}{R \tilde{\lambda}(t)} \lambda(t) \int (\int \chi^{2}(\frac{\tilde{\lambda}(t) y - \lambda(t) s}{R \lambda(t)}) (\frac{y \tilde{\lambda}(t) - s \lambda(t)}{\lambda(t)}) |\frac{1}{\lambda(t)^{1/2}} Iu(t,\frac{y}{\lambda(t)})|^{2} dy) \\ 
\times (\int \chi^{2}(\frac{\tilde{\lambda}(t) x - s \lambda(t)}{R \lambda(t)}) Im[\frac{1}{\lambda(t)^{1/2}} \bar{Iu}(t, \frac{x}{\lambda(t)}) \partial_{x}(\frac{1}{\lambda(t)^{1/2}} Iu(t, \frac{x}{\lambda(t)}))] dx) ds.
\endaligned
\end{equation}
\begin{remark}
Throughout these calculations, we understand that $\lambda^{1/2} Iu(\frac{x}{\lambda})$ refers to the rescaling of the function $Iu(x)$, not the $I$-operator acting on a rescaling of $u$.
\end{remark}

For any $\xi \in \mathbb{R}$,
\begin{equation}\label{4.33}
\aligned
= \frac{\dot{\tilde{\lambda}}(t)}{R \tilde{\lambda}(t)} \lambda(t) \int (\int \chi^{2}(\frac{\tilde{\lambda}(t) y - \lambda(t) s}{R \lambda(t)}) |\frac{e^{ix \xi}}{\lambda(t)^{1/2}} Iu(t,\frac{y}{\lambda(t)})|^{2} dy) \\
\times (\int \chi^{2}(\frac{\tilde{\lambda}(t) x - \lambda(t) s}{R \lambda(t)}) Im[\frac{e^{-ix \xi}}{\lambda(t)^{1/2}} \bar{Iu}(t, \frac{x}{\lambda(t)}) \partial_{x}(\frac{e^{ix \xi}}{\lambda(t)^{1/2}} Iu(t, \frac{x}{\lambda(t)})] (\frac{x \tilde{\lambda}(t) - s \lambda(t)}{\lambda(t)}) dx) ds \\
- \frac{\dot{\tilde{\lambda}}(t)}{R \tilde{\lambda}(t)} \lambda(t) \int (\int \chi^{2}(\frac{\tilde{\lambda}(t) y - \lambda(t) s}{R \lambda(t)}) (\frac{y \tilde{\lambda}(t) - s \lambda(t)}{\lambda(t)}) |\frac{e^{ix \xi}}{\lambda(t)^{1/2}} Iu(t,\frac{y}{\lambda(t)})|^{2} dy) \\ 
\times (\int \chi^{2}(\frac{\tilde{\lambda}(t) x - s \lambda(t)}{R \lambda(t)}) Im[\frac{e^{-ix \xi}}{\lambda(t)^{1/2}} \bar{Iu}(t, \frac{x}{\lambda(t)}) \partial_{x}(\frac{e^{ix \xi}}{\lambda(t)^{1/2}} Iu(t, \frac{x}{\lambda(t)}))] dx) ds.
\endaligned
\end{equation}
In particular, if we choose $\xi = \xi(s)$,
\begin{equation}\label{4.34}
\aligned
= \frac{\dot{\tilde{\lambda}}(t)}{R \tilde{\lambda}(t)} \lambda(t) \int (\int \chi^{2}(\frac{\tilde{\lambda}(t) y - \lambda(t) s}{R \lambda(t)}) |v_{s,t}|^{2} dy) \\ \times (\int \chi^{2}(\frac{\tilde{\lambda}(t) x - \lambda(t) s}{R \lambda(t)}) Im[\bar{v}_{s,t}(t, \frac{x}{\lambda(t)}) \partial_{x}(v_{s,t})] (\frac{x \tilde{\lambda}(t) - s \lambda(t)}{\lambda(t)}) dx) ds \\
= \frac{\dot{\tilde{\lambda}}(t)}{R} \int (\int \chi^{2}(\frac{\tilde{\lambda}(t) (y - s)}{R \lambda(t)}) |v_{s,t}|^{2} dy)  (\int \chi^{2}(\frac{\tilde{\lambda}(t) (x - s)}{R \lambda(t)}) Im[\bar{v}_{s,t}(t, \frac{x}{\lambda(t)}) \partial_{x}(v_{s,t})] (\frac{\tilde{\lambda}(t)(x - s)}{\lambda(t)}) dx) ds.
\endaligned
\end{equation}
Then by the Cauchy-Schwarz inequality,
\begin{equation}\label{4.35}
\aligned
\lesssim  \frac{\eta^{4}}{R} \lambda(t) \tilde{\lambda}(t)^{2} \int  (\int \chi^{2}(\frac{\tilde{\lambda}(t) (y - s)}{R \lambda(t)}) |v_{s,t}|^{2} dy) (\int \chi^{2}(\frac{\tilde{\lambda}(t) (x - s)}{R \lambda(t)}) |\partial_{x}(v_{s,t})|^{2} dx) ds \\
+ \frac{1}{\eta^{4}} \frac{|\dot{\tilde{\lambda}}(t)|^{2}}{\lambda(t) \tilde{\lambda}(t)^{2}} \frac{\tilde{\lambda}(t)}{R \lambda(t)} (\int \chi^{2}(\frac{\tilde{\lambda}(t) (y - s)}{R \lambda(t)}) |v_{s,t}|^{2} dy) \int \chi^{2}(\frac{\tilde{\lambda}(t) (x - s)}{R \lambda(t)}) |v_{s,t}|^{2} (\frac{\tilde{\lambda}(t)(x - s)}{\lambda(t)})^{2} dx)
\endaligned
\end{equation}

The first term in $(\ref{4.35})$ can be absorbed into $(\ref{4.28})$. The second term in $(\ref{4.35})$ is bounded by
\begin{equation}\label{4.36}
\frac{1}{\eta^{4}} \frac{|\dot{\tilde{\lambda}}(t)|^{2}}{\lambda(t) \tilde{\lambda}(t)^{2}} R^{2} \| u \|_{L_{t}^{\infty} L_{x}^{2}}^{4}.
\end{equation}
The smoothing algorithm from \cite{dodson2015global} is used to control this term. Recall that after $n$ iterations of the smoothing algorithm on an interval $[0, T]$, $\tilde{\lambda}(t)$ has the following properties:

\begin{enumerate}
\item $\tilde{\lambda}(t) \leq \lambda(t)$,

\item If $\dot{\tilde{\lambda}}(t) \neq 0$, then $\lambda(t) = \tilde{\lambda}(t)$,

\item $\tilde{\lambda}(t) \geq 2^{-n} \lambda(t)$,
 
\item $\int_{0}^{T} |\dot{\tilde{\lambda}}(t)| dt \leq \frac{1}{n} \int_{0}^{T} |\dot{\lambda}(t)| \frac{\tilde{\lambda}(t)}{\lambda(t)} dt$, with implicit constant independent of $n$ and $T$.
\end{enumerate}

Therefore,
\begin{equation}\label{4.37}
\int_{0}^{T_{n}} \frac{1}{\eta^{4}} \frac{|\dot{\tilde{\lambda}}(t)|^{2}}{\lambda(t) \tilde{\lambda}(t)^{2}} R^{2} \| u \|_{L_{t}^{\infty} L_{x}^{2}}^{4} dt \leq \frac{1}{\eta^{4}} \| u \|_{L_{t}^{\infty} L_{x}^{2}}^{4} \int_{0}^{T_{n}} \frac{|\dot{\lambda}(t)|}{\lambda(t)^{3}} R^{2} |\dot{\tilde{\lambda}}(t)| dt \lesssim \frac{1}{n} \frac{R^{2}}{\eta^{4}} \| u \|_{L_{t}^{\infty} L_{x}^{2}}^{4} \int_{0}^{T_{n}} \tilde{\lambda}(t)  \lambda(t)^{2} dt.
\end{equation}

Since $\sup_{t \in [0, T_{n}]} \lambda(t) \leq 2^{-2n} \int_{0}^{T_{n}} \lambda(t)^{3} dt$,
\begin{equation}\label{4.39}
R_{n} \sup_{t \in [0, T_{n}]} |M(t)| \lesssim R_{n} o(2^{2n}) \cdot \sup_{t \in [0, T]} \lambda(t).
\end{equation}
Therefore, it is possible to take a sequence $\eta_{n} \searrow 0$, $R_{n} \nearrow \infty$, probably very slowly, such that
\begin{equation}\label{4.38}
\frac{1}{n} \frac{R^{2}}{\eta^{4}} \| u \|_{L_{t}^{\infty} L_{x}^{2}}^{4} \int_{0}^{T_{n}} |\tilde{\lambda}(t)| \lambda(t)^{2} dt = o_{n}(1) \int_{0}^{T_{n}} \tilde{\lambda}(t) \lambda(t)^{2} dt,
\end{equation}
\begin{equation}\label{4.40}
R_{n} \sup_{t \in [0, T_{n}]} |M(t)| \lesssim o(2^{2n}) \cdot \sup_{t \in [0, T]} \lambda(t),
\end{equation}
\begin{equation}\label{4.40.1}
O(\eta_{n}^{4} \| u \|_{L_{t}^{\infty} L_{x}^{2}}^{2} \int_{0}^{T} \tilde{\lambda}(t) \| u(t) \|_{L^{6}}^{6} dt) \lesssim o_{n}(1) \int_{0}^{T_{n}} \tilde{\lambda}(t) \lambda(t)^{2} dt,
\end{equation}
and
\begin{equation}\label{4.40.2}
O(\frac{C(\eta_{n})}{R_{n}} \| u \|_{L_{t}^{\infty} L_{x}^{2}}^{2} \int_{0}^{T_{n}} \tilde{\lambda}(t) \| u(t) \|_{L_{x}^{6}}^{6} dt) \lesssim o_{n}(1) \int_{0}^{T_{n}} \tilde{\lambda}(t) \lambda(t)^{2} dt.
\end{equation}

Therefore, these terms may be safely treated as error terms, and repeating the analysis in $(\ref{3.17.1})$--$(\ref{3.36})$ for $(\ref{4.28})$, there exists a sequence of times $t_{n} \nearrow \sup(I)$ such that

\begin{equation}
E(\chi(\frac{(x - x(t_{n})) \tilde{\lambda}(t_{n})}{R_{n} \lambda(t_{n})}) v_{s_{n}, t_{n}}) \rightarrow 0,
\end{equation}
\begin{equation}
\| (1 - \chi(\frac{(x - x(t_{n})) \tilde{\lambda}(t_{n})}{R_{n} \lambda(t_{n})})) v_{s_{n}, t_{n}} \|_{L^{2}} \rightarrow 0,
\end{equation}
\begin{equation}
\| v_{s_{n}, t_{n}} \|_{L^{2}} \nearrow \| Q \|_{L^{2}},
\end{equation}
and
\begin{equation}\label{3.32}
\| \chi(\frac{(x - x(t_{n})) \tilde{\lambda}(t_{n})}{R_{n} \lambda(t_{n})}) Iv_{s_{n}, t_{n}} \|_{L^{6}} \sim 1.
\end{equation}

In this case as well, we can show that this sequence converges in $H^{1}$ to
\begin{equation}\label{3.36}
u_{0} = \lambda^{1/2} Q(\lambda(x - x_{0})).
\end{equation}
This proves Theorem $\ref{t2.2}$ for a general $\lambda(t)$. $\Box$

\section{Proof of Theorem $\ref{t1.3}$:} 
The proof of Theorem $\ref{t1.3}$ uses the argument used in the proof of Theorem $\ref{t1.2}$, combined with some reductions from \cite{fan20182}. First recall Lemma $4.2$ from \cite{fan20182}.
\begin{lemma}\label{l5.1}
Let $u$ be a solution to $(\ref{1.1})$ that satisfies the assumptions of Theorem $\ref{t1.3}$. Then there exists a sequence $t_{n} \nearrow T^{+}(u)$ such that $u(t_{n})$ admits a profile decomposition with profiles $\{ \phi_{j}, \{ x_{j,n}, \lambda_{j,n}, \xi_{j,n}, t_{j,n}, \gamma_{j,n} \} \}$, and there is a unique profile, call it $\phi_{1}$, such that
\begin{enumerate}
\item $\| \phi_{1} \|_{L^{2}} \geq \| Q \|_{L^{2}}$,

\item The nonlinear profile $\Phi_{1}$ associated to $\phi_{1}$ is an almost periodic solution in the sense of $(\ref{2.8})$ that does not scatter forward or backward in time.
\end{enumerate}
\end{lemma}

Now consider the nonlinear profile $\Phi_{1}$. To simplify notation relabel $\Phi_{1} = u$, and let $v_{s, t}$ be as in $(\ref{4.29})$. Using the same arguments as in the proof of Theorem $\ref{t1.2}$, there exists a sequence $t_{n} \nearrow T^{+}(u)$, $R_{n} \nearrow \infty$, $s_{n} \in \mathbb{R}$, $\tilde{\lambda}(t) \leq \lambda(t)$, such that
\begin{equation}\label{5.1}
E(\chi(\frac{(x - x(t_{n})) \tilde{\lambda}(t_{n})}{R_{n} \lambda(t_{n})}) v_{s_{n}, t_{n}}) \rightarrow 0,
\end{equation}
\begin{equation}\label{5.2}
\| (1 - \chi(\frac{(x - x(t_{n})) \tilde{\lambda}(t_{n})}{R_{n} \lambda(t_{n})})) v_{s_{n}, t_{n}} \|_{L^{2}} \rightarrow 0,
\end{equation}
\begin{equation}
\| v_{s_{n}, t_{n}} \|_{L^{2}} \nearrow \| u \|_{L^{2}},
\end{equation}
and
\begin{equation}\label{5.3}
\| \chi(\frac{(x - x(t_{n})) \tilde{\lambda}(t_{n})}{R_{n} \lambda(t_{n})}) I v_{s_{n}, t_{n}} \|_{L^{6}} \sim 1.
\end{equation}
Therefore, by the almost periodicity of $v$, there exists a sequence $g(t_{n})$ given by $(\ref{2.1.1})$ such that
\begin{equation}\label{5.4}
g(t_{n}) v(t_{n}) \rightarrow u_{0}, \qquad \text{in} \qquad L^{2},
\end{equation}
where $E(u_{0}) = 0$ and $\| u_{0} \|_{L^{2}} \geq \| Q \|_{L^{2}}$.\medskip

Next, utilize a blowup result of \cite{merle2003sharp}, \cite{merle2004universality}, \cite{merle2005blow} \cite{merle2006sharp}. We will state it here as it is stated in Theorem $3.1$ of \cite{fan20182}.
\begin{theorem}\label{t5.2}
Assume $u$ is a solution to $(\ref{1.1})$ with $\mu = -1$, and with initial data in $H^{1}$ that has non-positive energy and satisfies $(\ref{1.18})$. If $u$ is of zero energy, then $u$ blows up in finite time according to the log-log law,
\begin{equation}\label{5.5}
u(t,x) = \frac{1}{\lambda(t)^{1/2}} (Q + \epsilon)(\frac{x - x(t)}{\lambda(t)}) e^{i \gamma(t)}, \qquad x(t) \in \mathbb{R}, \qquad \gamma(t) \in \mathbb{R}, \lambda(t) > 0, \qquad \| \epsilon \|_{H^{1}} \leq \delta(\alpha),
\end{equation}
with the estimate
\begin{equation}\label{5.6}
\lambda(t) \sim \sqrt{\frac{T - t}{\ln|\ln(T - t)|}},
\end{equation}
and
\begin{equation}\label{5.7}
\lim_{t \rightarrow T} \int (|\nabla \epsilon(t,x)|^{2} + |\epsilon(t,x)|^{2} e^{-|x|}) dx = 0.
\end{equation}
\end{theorem}

Let $u$ be the solution to $(\ref{1.1})$ with initial data $u_{0}$. If $\| u_{0} \|_{L^{2}} = \| Q \|_{L^{2}}$ then we are done, using the analysis in the previous section. If $\| u_{0} \|_{L^{2}} > \| Q \|_{L^{2}}$, then Theorem $\ref{t5.2}$ implies that $u$ must be of the form $(\ref{5.5})$. Furthermore, by perturbative arguments, for any fixed $t' \in [0, T)$, $(\ref{5.4})$ implies that there exists a sequence $g(t_{n}, t')$ such that
\begin{equation}\label{5.8}
g(t_{n}, t') v(t_{n} + \frac{t'}{\lambda(t_{n})^{2}}) \rightarrow u(t'), \qquad \text{in} \qquad L^{2}.
\end{equation}
In fact, perturbative arguments also imply that there exists a sequence $t_{n}' \nearrow \infty$, perhaps very slowly, such that
\begin{equation}\label{5.9}
\| g(t_{n}, t_{n}') v(t_{n} + \frac{t_{n}'}{\lambda(t_{n})^{2}}) - u(t_{n}') \|_{L^{2}} \rightarrow 0.
\end{equation}
To see why this is so, first observe that if $g(t_{n}) v(t_{n}) = u_{0}$, then the uniqueness of solutions to $(\ref{1.1})$ combined with $(\ref{1.2})$ and $(\ref{1.15})$ implies that there exists $g(t_{n}, t')$ such that
\begin{equation}\label{5.9.1}
g(t_{n}, t') v(t_{n} + \frac{t'}{\lambda(t_{n})^{2}}) = u(t').
\end{equation}
Suppose that
\begin{equation}\label{5.9.2}
g(t_{n}) v(t_{n}) = u_{0} + w_{0}, \qquad \| w_{0} \|_{L^{2}} \ll 1.
\end{equation}
For $\tilde{T} \in [0, T)$ fixed, $\| u \|_{L_{t,x}^{6}([0, \tilde{T}] \times \mathbb{R})} < \infty$. Therefore, for $\| w_{0} \|_{L^{2}}$ sufficiently small, if $t' \in [0, \tilde{T}]$,
\begin{equation}\label{5.9.3}
g(t_{n}, t') v(t_{n} + \frac{t'}{\lambda(t_{n})^{2}}) = u(t') + w(t'),
\end{equation}
where $u$ solves $(\ref{1.1})$ and $w$ solves
\begin{equation}\label{5.9.4}
i w_{t} + \Delta w = -|u + w|^{4}(u + w) + |u|^{4} u.
\end{equation}
Partitioning $[0, \tilde{T}]$ into finitely many pieces such that $\| u \|_{L_{t,x}^{6}(I_{j} \times \mathbb{R})} \leq \epsilon$ on each piece, and iterating perturbative arguments on each piece, (see Chapter $1.3$ of \cite{dodson2019defocusing}), for $\| w_{0} \|_{L^{2}}$ sufficiently small,
\begin{equation}\label{5.9.5}
\| w \|_{L_{t}^{\infty} L_{x}^{2}([0, \tilde{T}] \times \mathbb{R})} \lesssim \exp(\frac{\| u \|_{L_{t,x}^{6}([0, \tilde{T}] \times \mathbb{R})}^{6}}{\epsilon^{6}}) \| w_{0} \|_{L^{2}}.
\end{equation}
Both $(\ref{5.8})$ and $(\ref{5.9})$ clearly follow from $(\ref{5.9.5})$, since as $\| w_{0} \|_{L^{2}}$, it is possible to take $\tilde{T} \nearrow T$ sufficiently slowly such that the right hand side of $(\ref{5.9.5})$ goes to zero.

Furthermore, Theorem $\ref{t5.2}$ implies that there exists a sequence $g(t_{n}')$ such that
\begin{equation}\label{5.10}
g(t_{n}') u(t_{n}') \rightharpoonup Q, \qquad \text{weakly in} \qquad L^{2}.
\end{equation}
Combining $(\ref{5.9})$ and $(\ref{5.10})$,
\begin{equation}\label{5.11}
g(t_{n}') g(t_{n}, t_{n}') v(t_{n} + \frac{t_{n}'}{\lambda(t_{n})^{2}}) \rightharpoonup Q, \qquad \text{weakly in} \qquad L^{2}.
\end{equation}
This completes the proof of Theorem $\ref{t1.3}$.

\section*{acknowledgements}
The author was supported on NSF grant DMS-$1764358$. The author also gratefully acknowledges the helpful comments and changes recommended by the anonymous referees.

\bibliography{biblio}

\begin{thebibliography}{10}

\bibitem{cazenave1989some}
Thierry Cazenave and Fred~B. Weissler.
\newblock Some remarks on the nonlinear {S}chr{\"o}dinger equation in the
  critical case.
\newblock In {\em Nonlinear semigroups, partial differential equations and
  attractors ({W}ashington, {DC}, 1987)}, volume 1394 of {\em Lecture Notes in
  Math.}, pages 18--29. Springer, Berlin, 1989.

\bibitem{cazenave1990cauchy}
Thierry Cazenave and Fred~B Weissler.
\newblock The {C}auchy problem for the critical nonlinear {S}chr{\"o}dinger
  equation in ${H}^{s}$.
\newblock {\em Nonlinear Analysis: Theory, Methods \& Applications},
  14(10):807--836, 1990.

\bibitem{dodson2015global}
Benjamin Dodson.
\newblock Global well-posedness and scattering for the mass critical nonlinear
  {S}chr{\"o}dinger equation with mass below the mass of the ground state.
\newblock {\em Advances in mathematics}, 285:1589--1618, 2015.

\bibitem{dodson2016global}
Benjamin Dodson.
\newblock Global well-posedness and scattering for the defocusing,
  ${L}^{2}$-critical, nonlinear {S}chr{\"o}dinger equation when d=1.
\newblock {\em American Journal of Mathematics}, 138(2):531--569, 2016.

\bibitem{dodson2019defocusing}
Benjamin Dodson.
\newblock {\em Defocusing {N}onlinear {S}chr{\"o}dinger {E}quations}, volume
  217.
\newblock Cambridge University Press, 2019.

\bibitem{fan20182}
Chenjie Fan.
\newblock The ${L}^{2}$ {W}eak {S}equential {C}onvergence of {R}adial
  {F}ocusing {M}ass {C}ritical {NLS} {S}olutions with {M}ass {A}bove the
  {G}round {S}tate.
\newblock {\em Int. Math. Res. Not. IMRN}, (7):4864--4906, 2021.

\bibitem{glassey1977blowing}
Robert~T Glassey.
\newblock On the blowing up of solutions to the {C}auchy problem for nonlinear
  {S}chr{\"o}dinger equations.
\newblock {\em Journal of Mathematical Physics}, 18(9):1794--1797, 1977.

\bibitem{merle1992uniqueness}
Frank Merle.
\newblock On uniqueness and continuation properties after blow-up time of
  self-similar solutions of nonlinear {S}chr{\"o}dinger equation with critical
  exponent and critical mass.
\newblock {\em Communications on pure and applied mathematics}, 45(2):203--254,
  1992.

\bibitem{merle1993determination}
Frank Merle.
\newblock Determination of blow-up solutions with minimal mass for nonlinear
  {S}chr{\"o}dinger equations with critical power.
\newblock {\em Duke Mathematical Journal}, 69(2):427--454, 1993.

\bibitem{merle2001existence}
Frank Merle.
\newblock Existence of blow-up solutions in the energy space for the critical
  generalized {K}d{V} equation.
\newblock {\em Journal of the American Mathematical Society}, 14(3):555--578,
  2001.

\bibitem{merle2003sharp}
Frank Merle and Pierre Raphael.
\newblock Sharp upper bound on the blow-up rate for the critical nonlinear
  {S}chr{\"o}dinger equation.
\newblock {\em Geometric \& Functional Analysis GAFA}, 13(3):591--642, 2003.

\bibitem{merle2004universality}
Frank Merle and Pierre Raphael.
\newblock On universality of blow-up profile for ${L}^{2}$-critical nonlinear
  {S}chr{\"o}dinger equation.
\newblock {\em Inventiones mathematicae}, 156(3):565--672, 2004.

\bibitem{merle2005blow}
Frank Merle and Pierre Raphael.
\newblock The blow-up dynamic and upper bound on the blow-up rate for critical
  nonlinear {S}chr{\"o}dinger equation.
\newblock {\em Annals of mathematics}, pages 157--222, 2005.

\bibitem{merle2006sharp}
Frank Merle and Pierre Raphael.
\newblock On a sharp lower bound on the blow-up rate for the ${L}^{2}$ critical
  nonlinear {S}chr{\"o}dinger equation.
\newblock {\em Journal of the American Mathematical Society}, 19(1):37--90,
  2006.

\bibitem{tao2008minimal}
Terence Tao, Monica Visan, and Xiaoyi Zhang.
\newblock Minimal-mass blowup solutions of the mass-critical {NLS}.
\newblock In {\em Forum Mathematicum}, volume~20, pages 881--919. De Gruyter,
  2008.

\bibitem{weinstein1983nonlinear}
Michael~I Weinstein.
\newblock Nonlinear {S}chr{\"o}dinger equations and sharp interpolation
  estimates.
\newblock {\em Communications in Mathematical Physics}, 87(4):567--576, 1983.

\end{thebibliography}
\bibliographystyle{plain}

\end{document}